\documentclass[11pt,reqno]{amsart}
 \usepackage{amssymb,amsfonts,amscd,amsbsy, color}
\usepackage[mathscr]{eucal}
\usepackage{url}

\newtheorem{theorem}{Theorem}[section]
\newtheorem{lemma}[theorem]{Lemma}

\newtheorem{corollary}[theorem]{Corollary}
\theoremstyle{definition}

\numberwithin{equation}{section}

\newcommand{\idem}{\textup{idem}}
\newcommand{\sgn}{\operatorname{sgn}}

\makeatletter
\def\imod#1{\allowbreak\mkern5mu({\operator@font mod}\,\,#1)}
\makeatother
\allowdisplaybreaks

\begin{document}

\title[On recursions for coefficients of mock theta functions]
{On recursions for coefficients of mock theta functions}

\author{Song Heng Chan}

\author{Renrong Mao}

\author{Robert Osburn}

\address{Division of Mathematical Sciences, School of Physical and Mathematical Sciences, Nanyang Technological University, 21 Nanyang link, Singapore, 637371, Republic of Singapore}

\address{School of Mathematical Sciences, Soochow University, SuZhou 215006, PR China}

\address{School of Mathematical Sciences, University College Dublin, Belfield, Dublin 4, Ireland}

\address{IH{\'E}S, Le Bois-Marie, 35, route de Chartres, F-91440 Bures-sur-Yvette, FRANCE}

\email{chansh@ntu.edu.sg}

\email{rrmao@suda.edu.cn}

\email{robert.osburn@ucd.ie, osburn@ihes.fr}


\date{\today}

\begin{abstract}
We use a generalized Lambert series identity due to the first author to present $q$-series proofs of recent results of Imamo{\u g}lu, Raum and Richter
concerning recursive formulas for the coefficients of two 3rd order mock theta functions. Additionally, we discuss an application of this identity to other mock theta functions.
\end{abstract}

\maketitle

\section{Introduction}

In \cite{chan1}, the first author proved the following generalized Lambert series identity:

\begin{align}  \nonumber
\frac{[a_1, \ldots, a_r]_\infty (q)_\infty^2}{[b_1, \ldots, b_s]_\infty} &=  \frac{[a_1/b_1, \ldots, a_r/b_1]_\infty}{[b_2/b_1,\ldots, b_s/b_1 ]_\infty} \sum_{k=-\infty}^{\infty}
\frac{(-1)^{(s-r)k}q^{(s-r)k(k+1)/2}}{1-b_1q^k} \left(
\frac{a_1\cdots a_r b_1^{s-r-1} }{b_2 \cdots b_s} \right)^k
\\
&\quad+ \idem(b_1; b_2, \ldots, b_s), \label{gls}
\end{align}

\noindent valid for nonnegative integers $r < s$. Here and throughout, we use the following standard $q$-hypergeometric notation

\begin{align}
(a)_{n} & := (a;q)_{n} := \prod_{k=0}^{n} (1-aq^{k-1}) \nonumber \\
(a_1, \ldots, a_m)_{n} & := (a_1, \ldots, a_m; q)_{n} := (a_1)_{n} \cdots (a_m)_{n} \nonumber \\
[a_1, \ldots, a_m]_{n} & := [a_1, \ldots, a_m; q]_{n} = (a_1, q/a_1, \ldots, a_m, q/a_m)_{n} \nonumber
\end{align}

\noindent valid for $n \in \mathbb{N} \cup \{\infty\}$ and $F(a_1, a_2, \ldots, a_m) + \idem(a_1; a_2, \ldots, a_n)$ to denote the sum

\begin{equation*}
\sum_{i=1}^{n} F(a_i, a_2, \ldots, a_{i-1}, a_1, a_{i+1}, \ldots, a_m)
\end{equation*}

\noindent where the $i$th term of the sum is obtained from the first by interchanging $a_1$ and $a_i$.

Identity (\ref{gls}) is of interest for several reasons. For example, it generalizes a key identity used by Atkin and Swinnerton-Dyer \cite{asd} in their proof of Dyson's conjectures on the rank of a partition. Also, (\ref{gls}) played a crucial role in the construction of quasimock theta functions \cite{blo} and rank-crank PDE's \cite{cdg},
in proving congruences for the mock theta function $\varphi(q)$ \cite{chan2}, Appell-Lerch sums \cite{chanmao1}, spt-type functions \cite{cjs1} and partition pairs \cite{cjs2} and in obtaining identities for generating functions  of other types of partition pairs \cite{chanmaopp} and various rank differences \cite{lo1}--\cite{mao10}.

Recently, Imamo{\u g}lu, Raum and Richter \cite{irr} proved some intriguing results concerning recursive formulas for the coefficients of the 3rd order mock theta functions

$$
f(q) := \sum_{n=0}^{\infty} \frac{q^{n^2}}{(-q)_{n}^2} = \sum_{n=0}^{\infty} c(f; n) q^n
$$

\noindent and

$$
\omega(q) := \sum_{n=0}^{\infty} \frac{q^{2n(n+1)}}{(q;q^2)_{n+1}^2} = \sum_{n=0}^{\infty} c(\omega; n) q^n.
$$

\noindent Namely, if $\sigma(n):=\sum_{0< d|n} d$, $\sgn^+(n):=\sgn(n)$ for $n\neq 0$ and $\sgn^+(0):=1$ and
\[
d(N,\tilde{N}, t, \tilde{t}) := \sgn^+ (N) \sgn^+(\tilde{N})(|N+t| - |\tilde{N}+\tilde{t}|),
\]

\noindent then we have the following (see Theorems 1 and 9 in \cite{irr}, slightly rewritten).

\begin{theorem} \label{t1}
For a fixed $n \in \mathbb{Z}^+$ and for any $a,b\in \mathbb{Z}$, set $N:= \frac{1}{6}(-3a+b-1)$ and $\tilde{N}:=\frac{1}{6}(3a+b-1)$. Then

\begin{equation} \label{t1id}
\sum_{\substack{m\in \mathbb{Z}\\3m^2+m\leq 2n}}
\left(m+\frac{1}{6}\right) c\left(f; n-\frac{3}{2}m^2-\frac{1}{2}m\right) =
\frac{4}{3} \sigma(n) -\frac{16}{3}\sigma\left(\frac{n}{2}\right)
-2 \sum_{\substack{a,b\in \mathbb{Z}\\ab=2n\\6|3a+b-1}} d\left(N, \tilde{N}, \frac{1}{6}, \frac{1}{6}\right).
\end{equation}

\noindent For a fixed $n\in \mathbb{Z}^+$ and for any  $a, b\in \mathbb{Z}$, set $N:=\frac{1}{12}(3a-b-2)$ and $\tilde{N}:=\frac{1}{12}(3a+b-4)$. Then

\begin{equation} \label{t9}
\sum_{\substack{m\in \mathbb{Z}\\3m^2+2m\leq n}}
\left(m+\frac{1}{3}\right) c\left(f; \frac{n}{2}-\frac{3}{2}m^2-m\right) =
-2 \sum_{\substack{a,b\in \mathbb{Z}\\ab=4n+1\\ 12|3a-b-2}} d\left(N, \tilde{N}, \frac{1}{6}, \frac{1}{3}\right).
\end{equation}

\end{theorem}

\begin{theorem} \label{t2}
For a fixed $n\in \mathbb{Z}^+$ and for any  $a, b\in \mathbb{Z}$, set $N:=\frac{1}{12}(3a-b-4)$ and $\tilde{N}:=\frac{1}{12}(3a+b-2)$. Then

\begin{equation} \label{t919}
\sum_{\substack{m\in \mathbb{Z}\\3m^2+2m\leq n}}
\left(m+\frac{1}{6}\right) c\left(\omega; n-3m^2-m\right) =
(-1)^{n+1} \sum_{\substack{a,b\in \mathbb{Z}\\ab=4n+3\\12|3a-b-4}} d\left(N,
\tilde{N}, \frac{1}{3}, \frac{1}{6}\right).
\end{equation}

\noindent For a fixed $n\in \mathbb{Z}^+$ and for any  $a, b\in \mathbb{Z}$,
set $N:=\frac{1}{6}(a-3b-2)$ and $\tilde{N}:=\frac{1}{6}(a+3b-2)$, and let
\[
R_n:=\left\{
       \begin{array}{ll}
\frac{2}{3}\left(\sigma\left(\frac{n}{4}\right) - \sigma\left(\frac{n}{2}\right)\right), & \hbox{if $n$ is even;} \vspace{.1in} \\

         \frac{1}{3}\sigma(n), & \hbox{if $n$ is odd.}
       \end{array}
     \right.
\]
Then
\begin{align} \label{t9201}
\sum_{\substack{m\in \mathbb{Z}\\3m^2+2m+1\leq n}}
\left(m+\frac{1}{3}\right) c\left(\frac{\omega(q)+\omega(-q)}{2}; n-3m^2-2m-1\right)
&=
R_n- \sum_{\substack{a,b\in \mathbb{Z}\\ab=n\\12|a-3b-8}} d\left(N,
\tilde{N}, \frac{1}{3}, \frac{1}{3}\right),
\intertext{and}
\label{t9202}
\sum_{\substack{m\in \mathbb{Z}\\3m^2+2m+1\leq n}}
\left(m+\frac{1}{3}\right) c\left(\frac{\omega(q)-\omega(-q)}{2}; n-3m^2-2m-1\right)
&=
-R_n+ \sum_{\substack{a,b\in \mathbb{Z}\\ab=n\\12|a-3b-2}} d\left(N,
\tilde{N}, \frac{1}{3}, \frac{1}{3}\right).
\end{align}
\end{theorem}

The identities (\ref{t1id})--(\ref{t9202}) were proven in \cite{irr} by applying holomorphic projection to the tensor product of a vector-valued harmonic weak Maass form of weight 1/2 and vector-valued modular form of weight 3/2 (see also \cite{dgo}, \cite{mertens}). In Remark 1, ii) of \cite{irr}, it was stated that these identities ``can sometimes also be furnished by Appell sums because these are typically expressible in terms of divisors. However, it is not clear whether Theorem 1 and 9 could be obtained using this idea".

Motivated by this remark, the main purpose of this paper is explain how (\ref{gls}) can be used to give a $q$-series proof of these identities. The idea is to compare the coefficients of $q^n$ in identities which express a modular form times either $f(q)$ or $\omega(q)$ (Appell-Lerch sums) in terms of Lambert series (divisor sums). Specifically, (\ref{t1id}) and (\ref{t9}) will basically follow from (\ref{c6}) and (\ref{idt9}), (\ref{t919}) from (\ref{t919a}), (\ref{t9201}) and (\ref{t9202}) from (\ref{pt920a}) and (\ref{c5}), respectively.

The paper is organized as follows. In Section 2, we discuss some preliminary $q$-series identities. In Section 3, we prove Theorems \ref{t1} and \ref{t2}. In Section 4, we discuss another application of (\ref{gls}) to other mock theta functions.

\section{Preliminaries}

To prove (\ref{t1id}), (\ref{t9201}) and (\ref{t9202}), we need the following two results.

\begin{lemma} We have
\begin{align}
\frac{(q)_\infty^2}{(-q)_\infty^2} \sum_{n=-\infty}^{\infty}
\frac{(-1)^{n}q^{n(3n+1)/2}}{1-xq^n}
&=4x\frac{(q)_\infty^2(q^2;q^2)_\infty^2}{[-x]_\infty[x^2;q^2]_\infty} + 2\sum_{n=-\infty}^{\infty}\frac{q^{n(3n+1)/2}}{(1+xq^n)^2} \nonumber \\
& + \sum_{n=-\infty}^{\infty}\frac{(6n-1)q^{n(3n+1)/2}}{1+xq^n}, \label{lat2} \\
q\frac{(q^2;q^2)_\infty^2}{(-q;q^2)_\infty^2} \sum_{n=-\infty}^{\infty} \frac{(-1)^{n}q^{3n(n+1)}}{1-xq^{2n}}
&= \frac{q}{x}\frac{(q^2;q^2)_\infty^4(-q;q^2)_\infty^2}{[x, -xq, -xq;q^2]_\infty} + \sum_{n=-\infty}^{\infty}\frac{q^{3n^2}}{(1+xq^{2n-1})^2} \nonumber \\
& + \sum_{n=-\infty}^{\infty}\frac{(3n-1)q^{3n^2}}{1+xq^{2n-1}}. \label{lat3}
\end{align}
\end{lemma}

\begin{proof}
Beginning with the case $r=0$, $s=3$ in \eqref{gls}, setting
$(b_1,b_2,b_3)=(x,-xq/a,-ax)$ and multiplying by $a[-a,
1/a^2]_\infty$, we find
\begin{align}  \nonumber
a\frac{(q)_\infty^2[-a,1/a^2]_\infty}{[x,-xq/a,-ax]_\infty}
&=
 \frac{[1/a^2]_\infty}{[-1/a]_\infty} \sum_{n=-\infty}^{\infty}
\frac{(-1)^{n}q^{n(3n+1)/2}}{1-xq^n}
+
 \sum_{n=-\infty}^{\infty}
\frac{q^{n(3n+1)/2}a^{-3n-1}}{1+xq^{n}/a}
\nonumber \\
& - \sum_{n=-\infty}^{\infty}
\frac{q^{n(3n+1)/2}a^{3n}}{1+axq^n}.
\label{at4}
\end{align}
Differentiating \eqref{at4} with respect to $a$ and letting $a\to
1$, we find that
\begin{align} \nonumber
4\frac{(q)_\infty^4(-q)_\infty^2}{[x, -xq, -x]_\infty}
&=
 \frac{(q)_\infty^2}{(-q)_\infty^2} \sum_{n=-\infty}^{\infty}
\frac{(-1)^{n}q^{n(3n+1)/2}}{1-xq^n}
-
\sum_{n=-\infty}^{\infty}\frac{q^{n(3n+1)/2}[1+3n(1+xq^n)]}{(1+xq^n)^2}
\\
&-\sum_{n=-\infty}^{\infty}\frac{q^{n(3n+1)/2}[1+(3n-1)(1+xq^n)]}{(1+xq^n)^2}.
\label{at8}
\end{align}
Rearranging and simplifying, we find that \eqref{at8} is
equivalent to \eqref{lat2}.

Again, for the case $r=0$, $s=3$ in \eqref{gls}, setting
$(b_1,b_2,b_3)=(x,-x/a,-ax)$, and multiplying by $[-a,
1/a^2]_\infty$, we find
\begin{align}  \nonumber
\frac{(q)_\infty^2[-a, 1/a^2]_\infty}{[x, -x/a, -ax]_\infty} &=
 \frac{[1/a^2]_\infty}{[-1/a]_\infty} \sum_{n=-\infty}^{\infty}
\frac{(-1)^{n}q^{3n(n+1)/2}}{1-xq^n} -
\sum_{n=-\infty}^{\infty}\frac{q^{3n(n+1)/2}a^{-3n-2}}{1+xq^n/a} \\
& +
\sum_{n=-\infty}^{\infty}\frac{q^{3n(n+1)/2}a^{3n+1}}{1+axq^n}.
\label{at3}
\end{align}Replacing $q$ by $q^2$ and setting $a=1/bq$ in \eqref{at3},  we find that
\begin{align} \nonumber
\frac{(q^2;q^2)_\infty^2[-1/(bq), b^2q^2;q^2]_\infty}{[x, -xbq, -x/(bq);q^2]_\infty}
&=
 \frac{[b^2q^2;q^2]_\infty}{[-bq;q^2]_\infty} \sum_{n=-\infty}^{\infty}
\frac{(-1)^{n}q^{3n(n+1)}}{1-xq^{2n}}
-
\sum_{n=-\infty}^{\infty}\frac{q^{3n^2-1}b^{3n-1}}{1+xbq^{2n-1}} \\
&+
\sum_{n=-\infty}^{\infty}\frac{q^{3n^2-1}b^{-3n-1}}{1+xq^{2n-1}/b}.
\label{at6}
\end{align}
Differentiating \eqref{at6}
with respect to $b$ and letting $b\to 1$, we arrive at
\begin{align} \nonumber
2\frac{(q^2;q^2)_\infty^4[-1/q;q^2]_\infty}{[x, -xq, -x/q;q^2]_\infty}
&=
2 \frac{(q^2;q^2)_\infty^2}{[-q;q^2]_\infty} \sum_{n=-\infty}^{\infty}
\frac{(-1)^{n}q^{3n(n+1)}}{1-xq^{2n}}
-
\sum_{n=-\infty}^{\infty}\frac{q^{3n^2-1}[1+(3n-2)(1+xq^{2n-1})]}{(1+xq^{2n-1})^2} \\
&-
\sum_{n=-\infty}^{\infty}\frac{q^{3n^2-1}[1+3n(1+xq^{2n-1})]}{(1+xq^{2n-1})^2}.
\label{at7}
\end{align}
Rearranging and simplifying, we find that \eqref{at7} is
equivalent to \eqref{lat3}.

\end{proof}

\begin{corollary} We have
\begin{align}
\frac{(q)_\infty^3}{(-q)_\infty^2} f(q) &=-4\sum_{n=1}^\infty
\frac{q^n}{(1-q^{n})^2} -16\sum_{n=1}^\infty
\frac{q^{2n}}{(1-q^{2n})^2}
+1+4\sum_{\substack{n=-\infty\\n\neq0}}^\infty\frac{q^{n(3n+1)/2}}{(1-q^{n})^2} \nonumber \\
&+ 2\sum_{\substack{n=-\infty\\ n\neq 0}}^\infty
\frac{(6n-1)q^{n(3n+1)/2}}{1-q^{n}}, \label{c6}
\\
\nonumber q\frac{(q^2;q^2)_\infty^3}{(-q;q^2)_\infty^2} \omega(-q)
&=\sum_{n=1}^\infty \frac{q^{2n-1}}{(1+q^{2n-1})^2}
-2\sum_{n=1}^\infty \frac{q^{2n}}{(1-q^{2n})^2}
+\sum_{\substack{n=-\infty\\n\neq0}}^\infty\frac{q^{3n^2}}{(1-q^{2n})^2} \nonumber \\
&+ \sum_{\substack{n=-\infty\\ n\neq 0}}^\infty
\frac{(3n-1)q^{3n^2}}{1-q^{2n}}. \label{c5}
\end{align}
\end{corollary}

\begin{proof}
\noindent For \eqref{c6}, multiply both sides of \eqref{lat2} by
$(1+x)(1+1/x)$, differentiate twice with respect to $x$, set
$x=-1$ and use

\begin{equation} \label{mockf}
f(q)= \frac{2}{(q)_\infty}\sum_{n=-\infty}^\infty\frac{(-1)^nq^{n(3n+1)/2}}{1+q^n}.
\end{equation}

\noindent Here, we have set $x=1$ in the identity (see (12.2.5) in \cite{ab})

\begin{equation*}
\sum_{n=0}^\infty \frac{q^{2n^2+2n}}{(xq;q^2)_{n+1}(q/x;q^2)_{n+1}} =\frac{1}{(q^2;q^2)_\infty} \sum_{n=-\infty}^\infty\frac{(-1)^nq^{3n^2+3n}}{1-xq^{2n+1}}.
\end{equation*}

For \eqref{c5}, multiply both sides of \eqref{lat3} by
$(1+q/x)(1+x/q)$, differentiate twice with respect to $x$, set
$x=-q$ and use

\begin{equation} \label{om00}
\omega(q) = \frac{1}{(q^2;q^2)_\infty}
\sum_{n=-\infty}^\infty\frac{(-1)^nq^{3n^2+3n}}{1-q^{2n+1}}.
\end{equation}

\noindent The latter follows from taking $x=-1$ in the identity
(see (12.2.3) in \cite{ab})

\begin{equation*}
\sum_{n=0}^\infty
\frac{q^{n^2}}{(xq;q)_n(q/x;q)_n}=\frac{(1-x)}{(q)_\infty}\sum_{n=-\infty}^\infty\frac{(-1)^nq^{n(3n+1)/2}}{1-xq^n}.
\end{equation*}
\end{proof}

To prove \eqref{t9}, \eqref{t9201} and \eqref{t9202}, we also need the following.
\begin{theorem}\label{t5}
For integers $l$ and $j,$ we have
\begin{align}
\frac{[-q^{l+j},q^{2l};q^{6l}]_\infty(q^{6l};q^{6l})^2_\infty}
{[-q^{j},q^{l},-q^{2l+j};q^{6l}]_\infty} \bigg\{j &+ 2
\sum_{m=0}^\infty\bigg(\frac{ q^{6lm+l}}{1-q^{6lm+l}}-\frac{ q^{6lm+5l}}{1-q^{6lm+5l}}+\frac{ q^{6lm+l+j}}{1+q^{6lm+l+j}} \nonumber \\
& -\frac{ q^{6lm+5l-j}}{1+q^{6lm+5l-j}} -2\frac{ q^{6lm+3l-j}}{1+q^{6lm+3l-j}}+2\frac{ q^{6lm+3l+j}}{1+q^{6lm+3l+j}}\bigg)\bigg\}\nonumber\\
&=\sum_{n=-\infty}^\infty \frac{(6n+2+j)(-1)^nq^{3ln(n+1)+ln+l}}{1+q^{6ln+j+2l}} \nonumber \\
&+\sum_{n=-\infty}^\infty \frac{(6n+j)(-1)^nq^{3ln(n+1)-ln}}{1+q^{6ln+j}} \nonumber\\
&+\frac{4(q^{2l};q^{2l})^2_\infty}
{(q^{l};q^{2l})^2_\infty}\sum_{n=-\infty}^\infty \frac{(-1)^nq^{3ln(n+1)+2ln+2l}}{1+q^{6ln+j+3l}}
\label{idt5}.
\end{align}
\end{theorem}

\begin{proof}
First, we need to prove the following:

\begin{align}
\frac{[-q^{l+j},q^{2l};q^{6l}]_\infty(q^{6l};q^{6l})^2_\infty}{[-q^{j},q^{l},-q^{2l+j};q^{6l}]_\infty} & \sum_{m=0}^\infty\left(\frac{ q^{6lm+l}}{1-q^{6lm+l}}-\frac{ q^{6lm+5l}}{1-q^{6lm+5l}}+\frac{ q^{6lm+l+j}}{1+q^{6lm+l+j}}-\frac{ q^{6lm+5l-j}}{1+q^{6lm+5l-j}}\right)\nonumber\\
&= \Biggl( 2\sum_{m=0}^\infty\left(\frac{ q^{6lm+l}}{1-q^{6lm+l}}-\frac{ q^{6lm+5l}}{1-q^{6lm+5l}}\right) \times\sum_{n=-\infty}^\infty \frac{(-1)^nq^{3ln(n+1)+ln+l}}{1+q^{6ln+j+2l}} \Biggr)
\nonumber\\
&+\sum_{n=-\infty}^\infty \frac{(n+1)(-1)^nq^{3ln(n+1)+ln+l}}{1+q^{6ln+j+2l}} + \sum_{n=-\infty}^\infty \frac{n(-1)^nq^{3ln(n+1)-ln}}{1+q^{6ln+j}}
\label{pt51}
\end{align}
and
\begin{align}
\frac{[-q^{l+j},q^{2l};q^{6l}]_\infty(q^{6l};q^{6l})^2_\infty} {[-q^{j},q^{l},-q^{2l+j};q^{6l}]_\infty} & \sum_{m=0}^\infty\left(\frac{ q^{6lm+3l-j}}{1+q^{6lm+3l-j}}-\frac{ q^{6lm+3l+j}}{1+q^{6lm+3l+j}}\right)\nonumber\\
&=\Biggl( \sum_{m=0}^\infty\left(\frac{ q^{6lm+l}}{1-q^{6lm+l}}-\frac{ q^{6lm+5l}}{1-q^{6lm+5l}}\right) \times\sum_{n=-\infty}^\infty \frac{(-1)^nq^{3ln(n+1)+ln+l}}{1+q^{6ln+j+2l}} \Biggr) \nonumber \\
& -\sum_{n=-\infty}^\infty \frac{n(-1)^nq^{3ln(n+1)+ln+l}}{1+q^{6ln+j+2l}} -\sum_{n=-\infty}^\infty \frac{n(-1)^nq^{3ln(n+1)-ln}}{1+q^{6ln+j}} \nonumber \\
&-\frac{[q^{-2l},q^{2l};q^{6l}]_\infty(q^{6l};q^{6l})^2_\infty}
{[q^{3l},q^{l},q^{-l};q^{6l}]_\infty}\sum_{n=-\infty}^\infty \frac{(-1)^nq^{3ln(n+1)+2ln+3l}}{1+q^{6ln+j+3l}}
\label{pt52}.
\end{align}

\noindent As the proofs of (\ref{pt51}) and (\ref{pt52}) are similar, we only give details for \eqref{pt52}. Replacing $q, a_1, a_2, b_1, b_2$ and $b_3$ by $q^{6l}, -q^{j+3l}, -q^{j+l}, -bq^{j+3l}, -q^{j+2l}$ and $-q^{j}$, respectively, in \eqref{gls} with $r=2, s=3$ and after rearranging, we obtain

\begin{align}
\frac{[-q^{3l+j},-q^{j+l}, q^{2l}, bq^{3l};q^{6l}]_\infty(q^{6l};q^{6l})^2_\infty} {[-q^{j+2l},-q^{j},q^3l,q^{l},-bq^{j+3l};q^{6l}]_\infty} &=\frac{b[1/b,q^{-2l}/b,q^{2l};q^{6l}]_\infty}{[q^3l,q^{l},q^{-l}/b;q^{6l}]_\infty} \sum_{n=-\infty}^\infty \frac{(-1)^{n+1}q^{3ln(n+1)+2ln+3l}}{1+bq^{6ln+j+3l}}\nonumber\\
&+ \frac{[q^{l},bq^{3l};q^{6l}]_\infty}{[q^{3l},bq^{l};q^{6l}]_\infty}\sum_{n=-\infty}^\infty \frac{(-1/b)^nq^{3ln(n+1)+ln+l}}{1+q^{6ln+j+2l}} \nonumber \\
& +\sum_{n=-\infty}^\infty \frac{(-1/b)^nq^{3ln(n+1)-ln}}{1+q^{6ln+j}}. \label{eastharlem}
\end{align}

\noindent Applying the operator $\frac{d}{db}\Big|_{b=1}$ on both sides of \eqref{eastharlem} and simplifying, we obtain \eqref{pt52}. If we multiply four times \eqref{pt52}, then subtract from twice \eqref{pt51}, we find

\begin{align}
\frac{[-q^{l+j},q^{2l};q^{6l}]_\infty(q^{6l};q^{6l})^2_\infty}{[-q^{j},q^{l},-q^{2l+j};q^{6l}]_\infty} & \sum_{m=0}^\infty\bigg\{2\left(\frac{ q^{6lm+l}}{1-q^{6lm+l}}-\frac{ q^{6lm+5l}}{1-q^{6lm+5l}}+\frac{ q^{6lm+l+j}}{1+q^{6lm+l+j}}-\frac{ q^{6lm+5l-j}}{1+q^{6lm+5l-j}}\right)\nonumber\\
& -4 \left(\frac{ q^{6lm+3l-j}}{1+q^{6lm+3l-j}}-\frac{ q^{6lm+3l+j}}{1+q^{6lm+3l+j}}\right)\bigg\}\nonumber\\
&=\sum_{n=-\infty}^\infty \frac{(6n+2)(-1)^nq^{3ln(n+1)+ln+l}}{1+q^{6ln+j+2l}}
+\sum_{n=-\infty}^\infty \frac{6n(-1)^nq^{3ln(n+1)-ln}}{1+q^{6ln+j}} \nonumber \\
& +\frac{4(q^{2l};q^{2l})^2_\infty} {(q^{l};q^{2l})^2_\infty}\sum_{n=-\infty}^\infty \frac{(-1)^nq^{3ln(n+1)+2ln+2l}}{1+q^{6ln+j+3l}}
\label{pt53}.
\end{align}

\noindent Replacing $q, a_1, b_1, b_2$ by $q^{6l}, -q^{l+j}, -q^j, -q^{j+2l}$ in \eqref{gls} with $r=1, s=2$, we obtain

\begin{align}\label{pt54}
\frac{[-q^{l+j},q^{2l};q^{6l}]_\infty(q^{6l};q^{6l})^2_\infty}
{[-q^{j},q^{l},-q^{2l+j};q^{6l}]_\infty}=
\sum_{n=-\infty}^\infty \frac{(-1)^nq^{3ln(n+1)+ln+l}}{1+q^{6ln+j+2l}}
+\sum_{n=-\infty}^\infty \frac{(-1)^nq^{3ln(n+1)-ln}}{1+q^{6ln+j}}
\end{align}

\noindent Finally, \eqref{pt53} plus $j$ times \eqref{pt54} yields \eqref{idt5}.
\end{proof}

To prove \eqref{t919}, we need the following result whose proof is analogous to that of Theorem \ref{t5} and thus is omitted.

\begin{theorem}\label{t53}
We have

\begin{align}
\frac{[q^{2l+j},q^{4l};q^{6l}]_\infty(q^{6l};q^{6l})^2_\infty}{[-q^{j},-q^{2l},-q^{4l+j};q^{6l}]_\infty}\bigg\{j&+\sum_{m=0}^\infty2\biggl(\frac{ q^{6lm+2l}}{1+q^{6lm+2l}}-\frac{ q^{6lm+4l}}{1+q^{6lm+4l}}-\frac{ q^{6lm+2l+j}}{1-q^{6lm+2l+j}}\nonumber \\
& +\frac{ q^{6lm+4l-j}}{1-q^{6lm+4l-j}}\biggr) -4 \left(\frac{ q^{6lm+4l+j}}{1-q^{6lm+4l+j}}-\frac{ q^{6lm+2l-j}}{1-q^{6lm+2l-j}}\right)\bigg\}\nonumber\\
&=-\sum_{n=-\infty}^\infty \frac{(6n+4+j)q^{3ln(n+1)+2ln+2l}}{1+q^{6ln+j+4l}} \nonumber \\
& +\sum_{n=-\infty}^\infty \frac{(6n+j)q^{3ln(n+1)-2ln}}{1+q^{6ln+j}} \nonumber \\
& +\frac{2(q^{2l};q^{2l})^2_\infty}{(-q^{2l};q^{2l})^2_\infty}\sum_{n=-\infty}^\infty \frac{(-1)^nq^{3ln(n+1)+2ln+2l}}{1-q^{6ln+j+4l}}. \nonumber
\end{align}
\end{theorem}

\section{Proofs of Theorems \ref{t1} and \ref{t2}}

\begin{proof}[Proof of Theorem \ref{t1}]

We begin with the proof of (\ref{t1id}). First, note that

\begin{align*}
\sum_{n=1}^\infty \frac{q^n}{(1-q^n)^2}  + 4 \sum_{n=1}^\infty \frac{q^{2n}}{(1-q^{2n})^2}
&=
\sum_{n=1}^\infty \sum_{m=1}^\infty m q^{mn}  + 4 \sum_{n=1}^\infty \sum_{m=1}^\infty  m q^{2mn},
\\
\sum_{\substack{n=-\infty\\n\neq0}}^\infty\frac{q^{n(3n+1)/2}}{(1-q^{n})^2}
&=
\sum_{n=1}^\infty\frac{q^{n(3n+1)/2}}{(1-q^{n})^2}
+\sum_{n=1}^\infty\frac{q^{n(3n-1)/2+2n}}{(1-q^{n})^2}\\
&=
\sum_{n=1}^\infty q^{n(3n+1)/2}\sum_{m=1}^\infty mq^{nm-n}
+\sum_{n=1}^\infty q^{n(3n+1)/2}\sum_{m=2}^\infty (m-1)q^{nm-n}
\\
&=
\sum_{m=1}^\infty (2m-1)\sum_{n=1}^\infty q^{n(3n+2m-1)/2},
\intertext{and}
\sum_{\substack{n=-\infty\\ n\neq 0}}^\infty
\frac{(6n-1)q^{n(3n+1)/2}}{1-q^{n}}
&=
\sum_{n=1}^\infty
\frac{(6n-1)q^{n(3n+1)/2}}{1-q^{n}}
+
\sum_{n=1}^\infty
\frac{(6n+1)q^{n(3n+1)/2}}{1-q^{n}}\\
&=
12\sum_{n=1}^\infty  nq^{n(3n+1)/2}\sum_{m=0}^\infty q^{mn}\\
&=
12\sum_{n=1}^\infty\sum_{m=1}^\infty  nq^{n(3n+2m-1)/2}.
\end{align*}

\noindent Therefore, \eqref{c6} is equivalent to

\begin{align}
\frac{(q)_\infty^3}{(-q)_\infty^2}f(q)
&=1-4\sum_{n=1}^\infty\sum_{m=1}^\infty nq^{mn}
-16\sum_{n=1}^\infty\sum_{m=1}^\infty nq^{2mn}
+4\sum_{n=1}^\infty \sum_{m=1}^\infty (6n+2m-1)q^{n(3n-1+2m)/2}. \label{e1}
\end{align}

\noindent By applying (see \cite[page 114, Entry 8 (ix)]{III})

\begin{equation} \label{qtp1}
\sum_{n=-\infty}^\infty (6n+1)q^{n(3n+1)/2} = \frac{(q)_\infty^3}{(-q)_\infty^2},
\end{equation}

\noindent and extracting the coefficient of $q^n$ on both sides of $\eqref{e1}$, we obtain

\begin{equation} \label{key1}
\sum_{\substack{m\in \mathbb{Z}\\3m^2+m\leq 2n}}\left(6m+1\right)
c\left(f; n-\frac{3}{2}m^2-\frac{1}{2}m\right)
=
-4\sigma(n)-16\sigma\left(\frac{n}{2}\right)
+
4\sum_{\substack{a,b\in \mathbb{Z}^+\\ab=2n\\b> 3a\\a\not\equiv b \pmod{2}}} (3a+b).
\end{equation}

\noindent We now examine the right-hand side of (\ref{t1id}). Observe that both $a$ and $b$ have the same sign, and when $a=3b$, $N$ is
not an integer. Hence, by splitting the sum according to the values of $a$ and $b$ in the range $1\leq b <3a$, $1< 3a<b$, $3a< b\leq -1$, and $b<3a<-1$, we find that

\begin{align}
\sum_{\substack{ a,b\in\mathbb{Z}\\ ab=2n\\6|3a+b-1}} d\left(N,
\tilde{N}, \frac{1}{6}, \frac{1}{6}\right)
&=\sum_{\substack{ a,b\in\mathbb{Z}^+\\ ab=2n\\ 1\leq b< 3a\\
6|3a+b-1}} \frac{b}{3}
\quad - \sum_{\substack{ a,b\in\mathbb{Z}^+\\ ab=2n\\  b > 3a \\
6|3a+b-1}} a
\quad - \sum_{\substack{ a,b\in\mathbb{Z}^-\\ ab=2n\\ 3a< b \leq -1\\
6|3a+b-1}} \frac{b}{3}
\quad + \sum_{\substack{ a,b\in\mathbb{Z}^-\\ ab=2n\\  b < 3a<-1 \\
6|3a+b-1}} a \nonumber
\\
&=\sum_{\substack{ a,b\in\mathbb{Z}^+\\ ab=2n\\ 1\leq b< 3a\\
6|3a+b-1\ \text{or} \ 6|3a+b+1}} \frac{b}{3}
\quad - \sum_{\substack{ a,b\in\mathbb{Z}^+\\ ab=2n\\  b > 3a \\
6|3a+b-1\ \text{or} \ 6|3a+b+1}} a \nonumber
\\
&=\frac{1}{3}\sum_{\substack{ a,b\in\mathbb{Z}^+\\ ab=2n\\ 1\leq b< 3a\\
b \not\equiv a \pmod{2}}} b
\quad - \sum_{\substack{ a,b\in\mathbb{Z}^+\\ ab=2n\\   b> 3a\\ b\not\equiv a \pmod{2}}} a, \label{split1}
\end{align}
where in the third equality, we note that
\[
\sum_{\substack{ a,b\in\mathbb{Z}^+\\ ab=2n\\ 1\leq b< 3a\\ 6|b+3a+3}} b
=
\sum_{\substack{ a,b\in\mathbb{Z}^+\\ ab=2n\\  1\leq b< 3a \\ 3|b\\ a\not\equiv b \pmod{2}}} b
=
\sum_{\substack{ a,b\in\mathbb{Z}^+\\ 3ab=2n\\  1\leq 3b< 3a \\ a\not\equiv b \pmod{2}}} 3b
=
\sum_{\substack{ a,b\in\mathbb{Z}^+\\ 3ba=2n\\  1\leq 3a < 3b \\ b\not\equiv a \pmod{2}}} 3a
=
\sum_{\substack{ a,b\in\mathbb{Z}^+\\ ab=2n\\   b> 3a\\ 3|b\\ a\not\equiv b \pmod{2}}} 3a
=
\sum_{\substack{a,b\in\mathbb{Z}^+\\ ab=2n\\   b> 3a \\ 6|b+3a+3}} 3a.
\]

\noindent Thus, by (\ref{key1}) and (\ref{split1}), it suffices to show that

\begin{equation}\label{t1irr}
\sum_{\substack{a,b\in \mathbb{Z}^+\\ab=2n\\b> 3a\\a\not\equiv b \pmod{2}}} b
= 3\sigma(n)-4\sigma\left(\frac{n}{2}\right)
-\sum_{\substack{ a,b\in\mathbb{Z}^+\\ ab=2n\\ 1\leq b< 3a\\ a\not\equiv b \pmod{2}}} b.
\end{equation}

\noindent By elementary manipulations, we see that
\begin{align*}
3\sigma(n)-4\sigma\left(\frac{n}{2}\right) &= \sum_{\substack{a,b\in \mathbb{Z}^+\\ab=n}} 3b \quad
- \sum_{\substack{a,b\in \mathbb{Z}^+\\ab=2n\\a\equiv b\equiv 0\pmod{2}}} 2b \\
&= \sum_{\substack{a,b\in \mathbb{Z}^+\\ab=n\\b\equiv 1 \pmod{2}}} b \quad + \sum_{\substack{a,b\in \mathbb{Z}^+\\ab=n\\a\equiv 1 \pmod{2}}} 2b \\
&= \sum_{\substack{a,b\in \mathbb{Z}^+\\ab=2n\\ a\equiv 0 \pmod{2}, \ b\equiv 1 \pmod{2}}} b \quad + \sum_{\substack{a,b\in \mathbb{Z}^+\\ab=2n\\ a\equiv 1 \pmod{2}, \  b\equiv 0 \pmod{2}}} b \\
&= \sum_{\substack{a,b\in \mathbb{Z}^+\\ab=2n\\a\not\equiv b \pmod{2}}} b
\end{align*}
and thus \eqref{t1irr} is true. This proves (\ref{t1id}).

Splitting the sum on the right-hand side of \eqref{t9} in a way similar to \eqref{split1}, we find that

\begin{align}
-6\sum_{\substack{a,b\in \mathbb{Z}\\ab=4n+1\\12|3a-b-2}} d\left(N,
\tilde{N}, \frac{1}{6}, \frac{1}{3}\right) &=\sum_{\substack{
a,b\in\mathbb{Z}^+\\ ab=4n+1\\ 1\leq b\leq 3a-2\\ 12|3a-b-2 \
\text{or} \ 12|3a-b+2}} b - \sum_{\substack{ a,b\in\mathbb{Z}^+\\
ab=4n+1\\   b\geq 3a+2\\ 12|3a-b-2\ \text{or} \ 12|3a-b+2}} 3a \label{split2} \\
& =:S_1(n)-3S_2(n). \nonumber
\end{align}
  The generating function for $S_1(n)$ is given by
\begin{equation} \label{s31}
\sum_{n=0}^\infty S_1(n)q^n
=\sum_{a=1}^\infty \sum_{\substack{b=1\\ 12|3a-b-2\\4|ab-1}}^{3a-2} bq^{(ab-1)/4}
+\sum_{a=1}^\infty \sum_{\substack{b=1\\ 12|3a-b+2\\4|ab-1}}^{3a-2} bq^{(ab-1)/4}.
\end{equation}
Note to have $12|3a-b-2$ and $4|ab-1$, we must have $a$ odd  and $b\equiv 1\pmod{6}$. Hence we
replace $a=2k+1$ and $b=6l+1$ in the first sum on the right-hand side of \eqref{s31}. Similarly, we replace
$a=2k+1$ and $b=6l+5$ in the second sum. This leads us to
\begin{align}
\sum_{n=0}^\infty S_1(n)q^n
&=
\sum_{k=0}^\infty \sum_{\substack{l=0\\ 2|k-l}}^{k} (6l+1)q^{3kl+(k+3l)/2}
+
\sum_{k=0}^\infty \sum_{\substack{l=0\\ 2|k-l}}^{k} (6l+5)q^{3kl+(5k+3l)/2+1} \nonumber
\\
&=
\sum_{l=0}^\infty \sum_{\substack{k=l\\ 2|k-l}}^\infty (6l+1)q^{3kl+(k+3l)/2}
+
\sum_{l=0}^\infty \sum_{\substack{k=l\\ 2|k-l}}^\infty (6l+5)q^{3kl+(5k+3l)/2+1} \nonumber
\\
&=
\sum_{l=0}^\infty(6l+1)q^{3l^2+2l} \sum_{\substack{k=0\\ 2|k}}^\infty q^{k(6l+1)/2}
+
\sum_{l=0}^\infty (6l+5)q^{3l^2+4l+1} \sum_{\substack{k=0\\ 2|k}}^\infty q^{k(6l+5)/2} \nonumber
\\
&=
\sum_{l=0}^\infty(6l+1)\frac{q^{3l^2+2l}}{1-q^{6l+1}}
+
\sum_{l=0}^\infty (6l+5)\frac{q^{3l^2+4l+1}}{1-q^{6l+5}} \nonumber \\
&=
\sum_{l=-\infty}^\infty\frac{(6l+1)q^{3l^2+2l}}{1-q^{6l+1}}, \label{s3gen}
\end{align}
where in the last equality, we replaced $l$ by $-l-1$ in the second sum. Similarly, the generating function for $S_2(n)$ is

\begin{equation} \label{s4gen}
\sum_{n=0}^{\infty} S_2(n) q^n = \sum_{n=-\infty}^\infty \frac{(2n-1)q^{3n^2+2n-2}}{1-q^{6n-3}}.
\end{equation}

\noindent From \eqref{mockf}, applying (see \cite[page 115, Entry 8 (x)]{III})

\begin{equation}  \label{qtp2}
\sum_{n=-\infty}^\infty (3n+1)q^{n(3n+2)} =\frac{(q^2;q^2)_\infty^3}{(-q;q^2)_\infty^2},
\end{equation}

\noindent \eqref{split2}, \eqref{s3gen} and \eqref{s4gen}, we see that \eqref{t9} follows upon extracting the coefficient of $q^n$ from both sides of

\begin{align}\label{idt9}
2\frac{(q^2;q^2)_\infty^2}{(-q;q^2)_\infty^2}\sum_{n=-\infty}^\infty
\frac{(-1)^nq^{n(3n+1)}}{1+q^{2n}} = \sum_{n=-\infty}^\infty
\frac{(6n+1)q^{3n^2+2n}}{1-q^{6n+1}}
-3\sum_{n=-\infty}^\infty \frac{(2n-1)q^{3n^2+2n-2}}{1-q^{6n-3}}.
\end{align}

\noindent Now, to prove (\ref{idt9}), we first set $l=1$ and $j=-3$ in Theorem \ref{t5} to obtain

\begin{align}\label{pt95}
\frac{[-q^{-2},q^{2};q^{6}]_\infty (q^{6};q^{6})^2_\infty}{[-q^{-3},q,-1/q;q^{6}]_\infty} \Biggl\{-1 & + 2\sum_{m=0}^\infty \Biggl(\frac{ q^{6m+1}}{1-q^{6m+1}}-\frac{ q^{6(m+1)-1}}{1-q^{6(m+1)-1}}+\frac{ q^{6m-2}}{1+q^{6m-2}}\nonumber \\
& -\frac{ q^{6(m+1)+2}}{1+q^{6(m+1)+2}}\Biggr) \Biggr\} \nonumber\\
&= \sum_{n=-\infty}^\infty \frac{(6n-1)(-1)^nq^{3n^2+4n+1}}{1+q^{6n-1}} +\sum_{n=-\infty}^\infty \frac{(6n-3)(-1)^nq^{3n^2+2n}}{1+q^{6n-3}}\nonumber \\
&+\frac{4[q^{-2},q^{2};q^{6}]_\infty(q^{6};q^{6})^2_\infty}{[q^{3},q,1/q;q^{6}]_\infty}\sum_{n=-\infty}^\infty \frac{(-1)^nq^{3n(n+1)+2n+3}}{1+q^{6n}}.
\end{align}

\noindent Replacing $q, a, b$ and  $c$ by $q^6, q^5, q^5$ and $-1/q^{2}$, respectively, in \cite[Corollary 3.2]{chan1}, we obtain

\begin{align*}
-1&+2\sum_{m=0}^\infty\left(\frac{ q^{6m+1}}{1-q^{6m+1}}-\frac{ q^{6(m+1)-1}}{1-q^{6(m+1)-1}}+\frac{ q^{6m-2}}{1+q^{6m-2}}-\frac{ q^{6(m+1)+2}}{1+q^{6(m+1)+2}}\right) \\
& =-\frac{[q^{10},-q^3,-q^3;q^6]_\infty(q^6;q^6)^2_\infty}{[q,q,-1/q^2,-q^8;q^6]_\infty},
\end{align*}

\noindent which together with \eqref{pt95} gives

\begin{align}\label{pt96}
&-\sum_{n=-\infty}^\infty \frac{(6n-1)(-1)^nq^{3n^2+4n+1}}{1+q^{6n-1}}
-\sum_{n=-\infty}^\infty \frac{(6n-3)(-1)^nq^{3n^2+2n}}{1+q^{6n-3}}
\nonumber\\&=-\frac{[q,q^2,q^2,-q^2,q^3;q^6]_\infty(q^6;q^6)^4_\infty}{[-q;q^6]_\infty}+4\frac{(q^2;q^2)_\infty^2}
{(q;q^2)_\infty^2}\sum_{n=-\infty}^\infty \frac{(-1)^nq^{n(3n+5)+2}}{1+q^{6n}}.
\end{align}

\noindent By \cite[Eq. (2.1)]{diss}, we have

\begin{align}
&-\frac{q^2[q,q^2,q^2,-q^2,q^3;q^6]_\infty(q^6;q^6)^4_\infty}{[-q;q^6]_\infty}
=-2\frac{(q^2;q^2)_\infty^2}{(q;q^2)_\infty^2}\sum_{n=-\infty}^\infty \frac{(-1)^nq^{n(3n+3)+2}}{1+q^{6n}}\label{pt9f}.
\end{align}

\noindent Substituting \eqref{pt9f} into \eqref{pt96}, we obtain

\begin{align}
&-\sum_{n=-\infty}^\infty \frac{(6n-1)(-1)^nq^{3n^2+4n+1}}{1+q^{6n-1}}
-\sum_{n=-\infty}^\infty \frac{(6n-3)(-1)^nq^{3n^2+2n}}{1+q^{6n-3}}
\nonumber\\&=4\frac{(q^2;q^2)_\infty^2}{(q;q^2)_\infty^2}\sum_{n=-\infty}^\infty \frac{(-1)^nq^{n(3n+5)+2}}{1+q^{6n}}-2\frac{(q^2;q^2)_\infty^2}{(q;q^2)_\infty^2}\sum_{n=-\infty}^\infty \frac{(-1)^nq^{n(3n+3)+2}}{1+q^{6n}}\nonumber\\&=2\frac{(q^2;q^2)_\infty^2}{(q;q^2)_\infty^2}\sum_{n=-\infty}^\infty \frac{(-1)^nq^{n(3n+1)+2}}{1+q^{2n}}, \label{t9f}
\end{align}

\noindent where the last step follows from the fact that

\begin{align*}
\sum_{n=-\infty}^\infty \frac{(-1)^nq^{n(3n+1)}}{1+q^{6n}}&=\sum_{n=-\infty}^\infty \frac{(-1)^nq^{n(3n+5)}}{1+q^{6n}}
\intertext{and}
\sum_{n=-\infty}^\infty \frac{(-1)^nq^{n(3n+1)}}{1+q^{2n}}&=\sum_{n=-\infty}^\infty \frac{(-1)^nq^{n(3n+1)}(1-q^{2n}+q^{4n})}{1+q^{6n}}.
\end{align*}

\noindent As \eqref{t9f} is equivalent to \eqref{idt9}, this proves (\ref{t9}).

\end{proof}

\begin{proof}[Proof of Theorem \ref{t2}]
For (\ref{t919}), we first split the sum on the right-hand side in a way similar to \eqref{split1} to obtain

\begin{align}\label{split3}
6(-1)^{n+1}\sum_{\substack{a,b\in \mathbb{Z}\\ab=4n+3\\12|3a-b-4}} d\left(N,
\tilde{N}, \frac{1}{3}, \frac{1}{6}\right) &= \sum_{\substack{
a,b\in\mathbb{Z}^+\\ ab=4n+3\\   b\geq 3a+2\\ 12|3a-b-4 \
\text{or} \ 12|3a-b+4}} 3a \quad -\sum_{\substack{ a,b\in\mathbb{Z}^+\\
ab=4n+3\\ 1\leq b\leq 3a-1\\ 12|3a-b-4 \ \text{or} \ 12|3a-b+4}} b.
\end{align}

\noindent By \eqref{om00}, \eqref{qtp1} and a calculation similar to (\ref{s3gen}) for the generating function of the right-hand side of \eqref{split3}, we see that \eqref{t919} follows extracting the coefficient of $q^n$ from both sides of

\begin{align}\label{t919a}
\frac{(q^2;q^2)_\infty^2}{(-q^2;q^2)_\infty^2}\sum_{n=-\infty}^\infty
\frac{(-1)^nq^{3n^2+3n}}{1-q^{2n+1}}
=-3\sum_{n=-\infty}^\infty\frac{(2n-1)q^{n(3n+1)-2}}{1+q^{6n-3}}
-\sum_{n=-\infty}^\infty\frac{(6n-1)q^{n(3n+1)-1}}{1+q^{6n-1}}.
\end{align}

\noindent To prove (\ref{t919a}), we set $j=-1, l=1$ in Theorem \ref{t53} to get
\begin{align}
\frac{[q,q^{4};q^{6}]_\infty(q^{6};q^{6})^2_\infty}
{[-1/q,-q^{2},-q^{3};q^{6}]_\infty}\bigg\{-1 &+
\sum_{m=0}^\infty2\left(\frac{ q^{6m+2}}{1+q^{6m+2}}-\frac{ q^{6m+4}}{1+q^{6m+4}}-\frac{ q^{6m+1}}{1-q^{6lm+1}}+\frac{ q^{6m+5}}{1-q^{6m+5}}\right)\bigg\}\nonumber\\&=-
\sum_{n=-\infty}^\infty \frac{(6n+3)q^{3n(n+1)+2n+2}}{1+q^{6n+3}}
+\sum_{n=-\infty}^\infty \frac{(6n-1)q^{3n(n+1)-2n}}{1+q^{6n-1}}
\nonumber\\&\quad+\frac{2(q^{2};q^{2})^2_\infty}
{(-q^{2};q^{2})^2_\infty}\sum_{n=-\infty}^\infty \frac{(-1)^nq^{3n(n+1)+2n+2}}{1-q^{6n+3}}
\label{leid54}.
\end{align}
Replacing $q, a, b$ and  $c$ by $q^6, q, q$ and $-q^{2}$, respectively, in \cite[Corollary 3.2]{chan1}, we obtain
\begin{align*}
-1&+\sum_{m=0}^\infty2\left(\frac{ q^{6m+2}}{1+q^{6m+2}}-\frac{ q^{6m+4}}{1+q^{6m+4}}-\frac{ q^{6m+1}}{1-q^{6lm+1}}+\frac{ q^{6m+5}}{1-q^{6m+5}}\right) \\
& =-\frac{[q^{2},-q^3,-q^3;q^6]_\infty(q^6;q^6)^2_\infty}{[q,q,-q^2,-q^4;q^6]_\infty},
\end{align*}
which together with \eqref{leid54} gives
\begin{align}
&\sum_{n=-\infty}^\infty \frac{(6n+3)q^{3n(n+1)+2n+2}}{1+q^{6n+3}}
-\sum_{n=-\infty}^\infty \frac{(6n-1)q^{3n(n+1)-2n}}{1+q^{6n-1}}
\nonumber\\&=\frac{2(q^{2};q^{2})^2_\infty}
{(-q^{2};q^{2})^2_\infty}\sum_{n=-\infty}^\infty \frac{(-1)^nq^{3n(n+1)+2n+2}}{1-q^{6n+3}}+\frac{q[q^{2},-q^3,q^{4};q^{6}]_\infty(q^{6};q^{6})^4_\infty}
{[q,-q,-q^2,-q^{2},-q^4;q^{6}]_\infty}
\label{leid55}.
\end{align}
 We note that (see \cite[Eq. (2.1)]{diss})
\begin{align}
\frac{q[q^{2},-q^3,q^{4};q^{6}]_\infty(q^{6};q^{6})^4_\infty}
{[q,-q,-q^2,-q^{2},-q^4;q^{6}]_\infty}
&=\frac{(q^2;q^2)_\infty^2}{(-q^2;q^2)_\infty^2}\sum_{n=-\infty}^\infty \frac{(-1)^nq^{n(3n+3)+1}}{1-q^{6n+3}}. \label{leid56}
\end{align}
Substituting \eqref{leid56} into \eqref{leid55}, we find that
\begin{align}\label{leid57}
&\sum_{n=-\infty}^\infty \frac{(6n+3)q^{3n(n+1)+2n+2}}{1+q^{6n+3}}
-\sum_{n=-\infty}^\infty \frac{(6n-1)q^{3n(n+1)-2n}}{1+q^{6n-1}}
\nonumber\\&=\frac{2(q^{2};q^{2})^2_\infty}
{(-q^{2};q^{2})^2_\infty}\sum_{n=-\infty}^\infty \frac{(-1)^nq^{3n(n+1)+2n+2}}{1-q^{6n+3}}+\frac{(q^2;q^2)_\infty^2}{(-q^2;q^2)_\infty^2}\sum_{n=-\infty}^\infty \frac{(-1)^nq^{3n(n+1)+1}}{1-q^{6n+3}}\nonumber\\&=\frac{(q^2;q^2)_\infty^2}{(-q^2;q^2)_\infty^2}\sum_{n=-\infty}^\infty \frac{(-1)^nq^{3n(n+1)+1}}{1-q^{2n+1}},
\end{align}
where the last step follows from the fact that
\begin{align*}
\sum_{n=-\infty}^\infty \frac{(-1)^nq^{3n(n+1)+2n+1}}{1-q^{6n+3}}&=\sum_{n=-\infty}^\infty \frac{(-1)^nq^{3n(n+1)+4n+2}}{1-q^{6n+3}}
\intertext{and}
\sum_{n=-\infty}^\infty \frac{(-1)^nq^{3n(n+1)}}{1-q^{2n+1}}&=\sum_{n=-\infty}^\infty \frac{(-1)^nq^{3n(n+1)}(1+q^{2n+1}+q^{4n+2})}{1-q^{6n+3}}.
\end{align*}
By \eqref{leid57}, we find that, to prove \eqref{t919a}, it suffices to show that
\begin{align*}
\sum_{n=-\infty}^\infty \frac{(6n+3)q^{3n(n+1)+2n+2}}{1+q^{6n+3}}=-\sum_{n=-\infty}^\infty \frac{(6n-3)q^{3n^2+n-1}}{1+q^{6n-3}}
\end{align*}
which is easily checked to be true by replacing $n$ by $-n$ in the sum on the left-hand side. This completes the proof of \eqref{t919a} and thus (\ref{t919}).

Now, by taking the sum and difference of \eqref{t9201} and \eqref{t9202}, respectively, we obtain the two equivalent formulas

\begin{align} \label{t920c}
\sum_{\substack{m\in \mathbb{Z}\\3m^2+2m+1\leq n}}
\left(m+\frac{1}{3}\right) c\left(\omega(q); n-3m^2-2m-1\right)
&= \sum_{\substack{a,b\in \mathbb{Z}\\ab=n\\12|a-3b-2}} d\left(N,
\tilde{N}, \frac{1}{3}, \frac{1}{3}\right) \nonumber \\
& - \sum_{\substack{a,b\in \mathbb{Z}\\ab=n\\12|a-3b-8}} d\left(N,
\tilde{N}, \frac{1}{3}, \frac{1}{3}\right),
\intertext{and}
\label{t920d}
\sum_{\substack{m\in \mathbb{Z}\\3m^2+2m+1\leq n}}
\left(m+\frac{1}{3}\right) c\left(\omega(-q); n-3m^2-2m-1\right)
&= 2R_n \nonumber \\
& -\sum_{\substack{a,b\in \mathbb{Z}\\ab=n\\12|a-3b-8\quad \text{or} \quad 12|a-3b-2}} d\left(N, \tilde{N}, \frac{1}{3}, \frac{1}{3}\right).
\end{align}

\noindent From \eqref{qtp2} and noting that

\begin{align*}
3\sum_{\substack{a,b\in \mathbb{Z}\\ab=n\\ 12|a-3b-8}}
d\left(N, \tilde{N}, \frac{1}{3}, \frac{1}{3}\right)
&= \sum_{\substack{ a,b\in\mathbb{Z}^+\\ ab=n\\ 1\leq a\leq 3b-1\\
12|a-3b-8 \ \text{or} \ 12|a-3b+8}} a - \sum_{\substack{
a,b\in\mathbb{Z}^+\\ ab=n\\   a\geq 3b+1\\ 12|a-3b-8 \ \text{or}
\ 12|a-3b+8}} 3b
\intertext{and}
3\sum_{\substack{a,b\in \mathbb{Z}\\ab=n\\12|a-3b-2}}
d\left(N, \tilde{N}, \frac{1}{3}, \frac{1}{3}\right)
&= \sum_{\substack{ a,b\in\mathbb{Z}^+\\ ab=n\\ 1\leq a\leq 3b-1\\
12|a-3b-2 \ \text{or} \ 12|a-3b+2}} a - \sum_{\substack{ a,b\in\mathbb{Z}^+\\
ab=n\\   a\geq 3b+1\\ 12|a-3b-2 \ \text{or} \ 12|a-b+2}} 3b,
\end{align*}

\noindent we see that \eqref{t920c} follows from

\begin{align}
\frac{q(q^2;q^2)_\infty^3}{(-q;q^2)_\infty^2}\omega(q)
 \nonumber
&= \sum_{n=-\infty}^\infty \frac{(3n+2)q^{3n^2+14n+8}}{1-q^{12n+8}}
-3\sum_{\substack{n=-\infty\\n\neq 0}}^\infty \frac{nq^{3n^2+2n}}{1-q^{12n}}
\\
&-\sum_{n=-\infty}^\infty \frac{(3n+2)q^{3n^2+8n+4}}{1-q^{12n+8}}
+3\sum_{\substack{n=-\infty\\n\neq 0}}^\infty \frac{nq^{3n^2+8n}}{1-q^{12n}}
\label{pt920a}
\end{align}

\noindent or equivalently

\begin{align} \label{equiv}
\frac{q(q^2;q^2)_\infty^3}{(-q;q^2)_\infty^2}\omega(q)
 \nonumber
&=
-3\sum_{\substack{n=-\infty\\n\neq 0}}^\infty \frac{nq^{3n^2+2n}}{1+q^{6n}}
-\sum_{n=-\infty}^\infty \frac{(3n+2)q^{3n^2+8n+4}}{1+q^{6n+4}}\\&=
-3\sum_{n=-\infty}^\infty \frac{nq^{3n^2+2n}}{1+q^{6n}}
+\sum_{n=-\infty}^\infty \frac{(3n+1)q^{3n^2+4n+1}}{1+q^{6n+2}}.
\end{align}

\noindent We now set $l=1, j=0$ and replace $q$ by $-q$ in Theorem \ref{t5} to obtain

\begin{align}\label{pom4}
&2\frac{[q,q^{2};q^{6}]_\infty(q^{6};q^{6})^2_\infty}
{[-1,-q,-q^{2};q^{6}]_\infty}
\sum_{m=0}^\infty\left(-\frac{ q^{6m+1}}{1+q^{6m+1}}+\frac{ q^{6m+5}}{1+q^{6m+5}}-\frac{ q^{6m+1}}{1-q^{6m+1}}+\frac{ q^{6m+5}}{1-q^{6m+5}}\right)
\nonumber\\&=-\sum_{n=-\infty}^\infty \frac{(6n+2) q^{3n^2+4n+1}}{1+q^{6n+2}}
+\sum_{n=-\infty}^\infty \frac{6n q^{3n^2+2n}}{1+q^{6n}}
+4\frac{(q^2;q^2)_\infty^2}{(-q;q^2)_\infty^2}\sum_{n=-\infty}^\infty \frac{(-1)^nq^{3n^2+5n+2}}{1-q^{6n+3}}.
\end{align}

\noindent By \cite[Corollary 3.1]{chan1}, we have

\begin{align*}
\sum_{m=0}^\infty\left(-\frac{ q^{6m+1}}{1+q^{6m+1}}+\frac{ q^{6m+5}}{1+q^{6m+5}}-\frac{ q^{6m+1}}{1-q^{6m+1}}+\frac{ q^{6m+5}}{1-q^{6m+5}}\right)
& =-4\sum_{j=-\infty}^\infty\frac{ q^{6j+1}}{1-q^{12j+2}} \\
& = -\frac{4q [q^8; q^{12}]_{\infty} (q^{12}; q^{12})_{\infty}^2}{[q^2, q^6; q^{12}]_{\infty}}
\end{align*}

\noindent which together with \eqref{pom4} gives

\begin{align*}
\sum_{n=-\infty}^\infty \frac{(6n+2)q^{3n^2+4n}}{1+q^{6n+2}}-
\sum_{n=-\infty}^\infty \frac{6nq^{3n^2+2n-1}}{1+q^{6n}}
& =2\frac{(q^2;q^2)_\infty^3}{(-q;q^2)_\infty^2}\times \frac{[q^6; q^{18}]_{\infty}^3 (q^{18}; q^{18})_{\infty}^3}{[q^2, q^3; q^6]_{\infty} (q^6; q^6)_{\infty}^2} \\
& +4\frac{(q^2;q^2)_\infty^2}{(-q;q^2)_\infty^2}\sum_{n=-\infty}^\infty \frac{(-1)^nq^{n(3n+5)+1}}{1-q^{6n+3}}.
\end{align*}

\noindent This together with \cite[Eq. (5.8)]{Hi-Mo1} implies

\begin{align}
2\frac{(q^2;q^2)_\infty^3}{(-q;q^2)_\infty^2}\omega(q)
= \sum_{n=-\infty}^\infty \frac{(6n+2)q^{3n^2+4n}}{1+q^{6n+2}}
-\sum_{n=-\infty}^\infty \frac{6nq^{3n^2+2n-1}}{1+q^{6n}} \nonumber
\end{align}

\noindent which is equivalent to (\ref{equiv}). Thus, (\ref{t920c}) is proven.

For \eqref{t920d}, it suffices to prove
\begin{align} \nonumber
\sum_{\substack{m\in \mathbb{Z}\\3m^2+2m+1\leq n}}&
\left(3m+1\right) c\left(\omega(-q); n-3m^2-2m-1\right)
\\ \nonumber
&=
6R_n
- \sum_{\substack{ a,b\in\mathbb{Z}^+\\ ab=n\\ 1\leq a\leq 3b-1\\
a-3b \equiv 2, 4, 8, 10 \pmod{12}}} a
+ \sum_{\substack{
a,b\in\mathbb{Z}^+\\ ab=n\\   a\geq 3b+1\\a-3b \equiv 2, 4, 8, 10 \pmod{12}}} 3b
\\ \nonumber
&=
6R_n
- \sum_{\substack{ a,b\in\mathbb{Z}^+\\ ab=n\\ 1\leq a\leq 3b-1\\
a\equiv b \pmod{2}}} a
+ \sum_{\substack{
a,b\in\mathbb{Z}^+\\ ab=n\\   a\geq 3b+1\\a \equiv b \pmod{2}}} 3b.
\\
&=
6R_n
- \sum_{\substack{ a,b\in\mathbb{Z}^+\\ ab=n\\ 1\leq a\leq 3b\\
a\equiv b \pmod{2}}} a
+ \sum_{\substack{
a,b\in\mathbb{Z}^+\\ ab=n\\   a\geq 3b\\a \equiv b \pmod{2}}} 3b, \label{pt920b}
\end{align}
where in the second equality, we used the fact that
\[
\sum_{\substack{ a,b\in\mathbb{Z}^+\\ ab=n\\ 1\leq a\leq 3b-1\\
a-3b \equiv 0 \pmod{6}}} a
=
\sum_{\substack{ a,b\in\mathbb{Z}^+\\ ab=n\\ 1\leq 3a\leq 3b-1\\
\\
a \equiv b \pmod{2}}} 3a
=
\sum_{\substack{
a,b\in\mathbb{Z}^+\\ ba=n\\   1\leq 3b\leq 3a-1\\
a \equiv b \pmod{2}}} 3b
=
\sum_{\substack{
a,b\in\mathbb{Z}^+\\ ab=n\\   a\geq 3b+1\\3|a\\
a \equiv b \pmod{2}}} 3b
=
\sum_{\substack{
a,b\in\mathbb{Z}^+\\ ab=n\\   a\geq 3b+1\\a-3b \equiv 0 \pmod{6}}} 3b.
\]

\noindent Next, we examine the right-hand side of \eqref{c5}. Note that
\begin{align*}
\sum_{n=1}^\infty \frac{q^{2n-1}}{(1+q^{2n-1})^2}
-2\sum_{n=1}^\infty \frac{q^{2n}}{(1-q^{2n})^2} &=
\sum_{n=1}^\infty \left(\frac{(2n-1)q^{2n-1}}{1-q^{4n-2}}-
\frac{4nq^{2n}}{1-q^{2n}} +\frac{2nq^{4n}}{1-q^{4n}} \right)
\\
&= \sum_{n=1}^\infty 6\tilde{R}_n q^n,
\end{align*}
where
\[
\tilde{R}_n :=\left\{
\begin{array}{ll}
    \frac{1}{3}(\sigma(\frac{n}{4})-2\sigma(\frac{n}{2})), & \hbox{if $n$ is even;} \vspace{.1in} \\
    \frac{1}{6}\sigma(n), & \hbox{if $n$ is odd,} \\
\end{array}
\right.
\]
and
\begin{align*}
\sum_{\substack{n=-\infty\\n\neq 0}}^\infty \frac{q^{3n^2}}{(1-q^{2n})^2}
&=\sum_{n=1}^\infty \frac{q^{3n^2}}{(1-q^{2n})^2} +\sum_{n=1}^\infty \frac{q^{3n^2+4n}}{(1-q^{2n})^2}
\\
&=\sum_{n=1}^\infty\sum_{m=0}^\infty q^{3n^2+2mn}(m+1) +\sum_{n=1}^\infty\sum_{m=1}^\infty q^{3n^2+2mn}(m-1)
\\
&=\sum_{n=1}^\infty q^{3n^2} +\sum_{n=1}^\infty\sum_{m=1}^\infty 2m q^{n(3n+2m)}.
\end{align*}
Similarly,
\begin{align*}
\sum_{\substack{n=-\infty\\n\neq 0}}^\infty \frac{(3n-1)q^{3n^2}}{1-q^{2n}}
&=\sum_{n=1}^\infty (3n-1)q^{3n^2} +\sum_{n=1}^\infty\sum_{m=1}^\infty 6n q^{n(3n+2m)}.
\end{align*}

\noindent Hence, identity \eqref{c5} is equivalent to

\begin{align}
q\frac{(q^2;q^2)_\infty^3}{(-q;q^2)_\infty^2} \omega(-q)
& = \sum_{n=1}^\infty 6\tilde{R}_n q^n +\sum_{n=1}^\infty 3nq^{3n^2}
    + \sum_{n=1}^\infty\sum_{m=1}^\infty (6n+2m) q^{n(3n+2m)}
\nonumber \\
&= \sum_{n=1}^\infty 6\tilde{R}_n q^n
   + \sum_{n=1}^\infty\sum_{m=1}^\infty (3n+2m) q^{n(3n+2m)}+\sum_{n=1}^\infty\sum_{m=0}^\infty 3nq^{n(3n+2m)} \nonumber \\
&= \sum_{n=1}^\infty 6\tilde{R}_n q^n
  +\sum_{n=1}^\infty q^n\sum_{\substack{b,m\in\mathbb{Z}^+\\b(3b+2m)=n}} (3b+2m)
  +\sum_{n=1}^\infty q^n\sum_{\substack{b,m\in\mathbb{Z}^+\\b(3b+2m-2)=n}} 3b \nonumber \\
&= \sum_{n=1}^\infty 6\tilde{R}_n q^n
  +\sum_{n=1}^\infty q^n\sum_{\substack{a,b\in\mathbb{Z}^+\\ab=n\\a>3b\\a\equiv b \pmod{2}}} a
  +\sum_{n=1}^\infty q^n\sum_{\substack{a,b\in\mathbb{Z}^+\\ab=n\\a>3b-2\\a\equiv b \pmod{2}}} 3b. \label{e2}
\end{align}

\noindent Applying \eqref{qtp2} while extracting the coefficient of $q^n$ from both sides of \eqref{e2},
we obtain

\begin{align} \nonumber
\sum_{\substack{m\in \mathbb{Z}\\3m^2+2m+1\leq n}}&
\left(3m+1\right) c\left(\omega(-q); n-3m^2-2m-1\right) =
6\tilde{R}_n
+\sum_{\substack{a,b\in\mathbb{Z}^+\\ab=n\\a>3b\\a\equiv b
\pmod{2}}} a
+\sum_{\substack{a,b\in\mathbb{Z}^+\\ab=n\\a>3b-2\\a\equiv b
\pmod{2}}} 3b.
\end{align}

\noindent Note that

\begin{align*}
\sum_{n=1}^\infty 6(R_n-\tilde{R}_n)q^n &=
\sum_{n=1}^\infty 2\sigma(n/2)q^{2n} +
\sum_{n=1}^\infty \sigma(2n-1)q^{2n-1}
\\
&=
\sum_{n=1}^\infty \frac{2nq^{4n}}{1-q^{4n}}
+ \sum_{n=1}^\infty \frac{(2n-1)q^{2n-1}}{1-q^{4n-2}}\\
&=\sum_{n=1}^\infty\sum_{m=1}^\infty 2nq^{4mn} +\sum_{n=1}^\infty\sum_{m=1}^\infty (2n-1)q^{(2n-1)(2m-1)}
\\
&=\sum_{n=1}^\infty q^n \sum_{\substack{a,b\in \mathbb{Z}^+\\ 4ab=n}} 2a
+\sum_{n=1}^\infty q^n \sum_{\substack{a,b\in \mathbb{Z}^+\\ (2a-1)(2b-1)=n}} (2a-1)
\\
&=\sum_{n=1}^\infty q^n \sum_{\substack{a,b\in \mathbb{Z}^+\\ ab=n\\a\equiv b \pmod{2}}} a.
\end{align*}

\noindent Thus, (\ref{pt920b}) follows upon observing that

\[
\sum_{\substack{a,b\in \mathbb{Z}^+\\ ab=n\\a\equiv b \pmod{2}}} a
- \sum_{\substack{ a,b\in\mathbb{Z}^+\\ ab=n\\ 1\leq a\leq 3b\\
a\equiv b \pmod{2}}} a
+ \sum_{\substack{
a,b\in\mathbb{Z}^+\\ ab=n\\   a\geq 3b\\a \equiv b \pmod{2}}} 3b
=\sum_{\substack{a,b\in\mathbb{Z}^+\\ab=n\\a>3b\\a\equiv b \pmod{2}}} a
+ \sum_{\substack{a,b\in\mathbb{Z}^+\\ab=n\\a>3b-2\\a\equiv b \pmod{2}}} 3b.
\]

\noindent This proves (\ref{t920d}). Adding (\ref{t920c}) and (\ref{t920d}), then dividing by two yields (\ref{t9201}) while subtracting (\ref{t920d}) from (\ref{t920c}), then dividing by two implies (\ref{t9202}).

\end{proof}

\section{Other applications of (\ref{gls})}

It is worth noting that (\ref{gls}) can also be used to obtain other identities involving mock theta functions. For example, in \cite{arz}, Andrews, Rhoades and Zwegers consider the automorphic properties of the $q$-series

\begin{equation*}
\nu_{2}(q) := \frac{1}{(q)_{\infty}^3} \left( \sum_{\substack{n=-\infty \\ n \neq 0}}^{\infty}  \frac{(-1)^{n+1} n q^{\frac{n(n+1)}{2}}}{1-q^n} - \frac{1}{4} - 2 \sum_{n=1}^{\infty} \frac{q^n}{(1+q^n)^2} \right),
\end{equation*}

\noindent which is related to the generating function for the number of concave compositions of $n$ \cite{and}. In particular, to show that $\nu_2(q)$ is a mock theta function, they require the following key identity (see Theorem 1.3 in \cite{arz})

\begin{equation} \label{ftonu}
\tilde{F}(q):= \frac{1}{(q)_{\infty} (-q)_{\infty}^2} \sum_{n=-\infty}^{\infty} \frac{q^{\frac{n(n+1)}{2}}}{1+q^n} = -2\nu_{2}(q).
\end{equation}

\noindent We now prove a generalization of (\ref{ftonu}).

\begin{theorem} \label{genftonu}
We have
 \begin{align}\label{l31}
&\frac{(q)^2_\infty}{[1/b]_\infty}\sum_{k=-\infty}^\infty\frac{(-b)^kq^{\frac{k(k+1)}{2}}}{1-bq^k}=
- \sum_{n=0}^\infty
\left(\frac{q^nb}{(1-bq^n)^2}+\frac{q^{n+1}/b}{(1-q^{n+1}/b)^2}\right)+ \sum_{\substack{n=-\infty \\ n \neq 0}}^{\infty} \frac{(-1)^nnq^{n(n+1)/2}}{1-q^n}.
\end{align}
\end{theorem}

\begin{proof}[Proof of Theorem \ref{genftonu}]
Setting $r=1$, $s=2$ and replacing $a_1$ by $bb_1$ in \eqref{gls},
after rearranging, we obtain

\begin{align}
\frac{[bb_1,b_1/b]_\infty(q)_\infty^2}{[b,b_1,b_1]_\infty}
=\sum_{k=-\infty}^\infty\frac{(-b)^kq^{\frac{k(k+1)}{2}}}{1-bq^k}-
\frac{[b]_\infty}{b[b_1]_\infty}\sum_{k=-\infty}^\infty\frac{(-b_1)^{k+1}q^{\frac{k(k+1)}{2}}}{1-b_1q^k} \label{last11}.
\end{align}

\noindent Multiplying by $(1-b_1)^2$ and applying the operator $\frac{d^2}{d^2b_1}\Big|_{b_1=1},$ we obtain

\begin{align}\label{30}
&\frac{(qb,q/b)_\infty}{(1-b)(q)^2_\infty}+\frac{[1/b]_\infty}{(q)^2_\infty}
\sum_{n=1}^\infty
\left(\frac{2q^{n}}{(1-q^n)^2}-\frac{q^nb}{(1-bq^n)^2}-\frac{q^n/b}{(1-q^n/b)^2}\right)
\nonumber\\&=\sum_{k=-\infty}^\infty\frac{(-b)^kq^{\frac{k(k+1)}{2}}}{1-bq^k}+
\frac{[b]_\infty}{b(q)^2_\infty} \sum_{\substack{n=-\infty \\ n\neq0}}^\infty \left\{
\frac{(-1)^nq^{n(n+1)/2}(1+q^n)}{(1-q^n)^2}+\frac{(-1)^nnq^{n(n+1)/2}}{1-q^n}\right\}.
\end{align}

\noindent Letting $a, c, d, e\rightarrow 1$ and $b\rightarrow q$ in a limiting case of  Watson's $_8\phi_7$ transformation, \cite[Eq.~(7.2), p.~16]{{III}}

\begin{equation*}
\sum_{n=0}^\infty \frac{(aq/bc,d,e;q)_n(\frac{aq}{de})^n}{(q,aq/b,
aq/c;q)_n} =\frac{(aq/d, aq/e;q)_\infty}{(aq, aq/de;q)_\infty}
\sum_{n=0}^\infty \frac{(a, b, c, d,
e;q)_n(1-aq^{2n})(-a^2)^{n}q^{n(n+3)/2}} {(q, aq/b, aq/c, aq/d,
aq/e;q)_n(1-a)(bcde)^n},
\end{equation*}

\noindent we find that

\begin{align}\label{waston}
\sum_{n=1}^\infty \frac{q^n}{(1-q^n)^2}=- \sum_{n=1}^\infty
\frac{(-1)^nq^{n(n+1)/2}(1+q^n)}{(1-q^n)^2}.
\end{align}

\noindent Substituting \eqref{waston} into \eqref{30} and rearranging, we obtain

\begin{align*}\label{31}
\frac{-1}{b(1-b)^2}- \sum_{n=1}^\infty
\left(\frac{q^nb}{(1-bq^n)^2}+\frac{q^n/b}{(1-q^n/b)^2}\right)
&=\frac{(q)^2_\infty}{[1/b]_\infty}\sum_{k=-\infty}^\infty\frac{(-b)^kq^{\frac{k(k+1)}{2}}}{1-bq^k} \\
& -\sum_{\substack{n=-\infty \\ n\neq0}}^\infty\frac{(-1)^nnq^{n(n+1)/2}}{1-q^n}
\end{align*}
which implies \eqref{l31}.
\end{proof}

Multiplying by $\frac{1-1/b}{(q;q)_\infty^3}$ on both sides of \eqref{l31} and setting $b=-1$, we obtain \eqref{ftonu}. 
Using Theorem \ref{genftonu}, one can also show

\begin{equation}\label{r}
q(q^4;q^4)^3_\infty B(q)= \sum_{n=0}^\infty \left(\frac{q^{4n+1}}{(1-q^{4n+1})^2}+\frac{q^{4n+3}}{(1-q^{4n+3})^2}\right)-
\sum_{\substack{n=-\infty \\ n\neq0}}^\infty\frac{(-1)^nnq^{2n(n+1)}}{1-q^{4n}}
\end{equation}

\noindent where

\begin{equation*}
B(q):=\sum_{n=0}^\infty\frac{q^n(-q;q^2)_n}{(q;q^2)_{n+1}}
\end{equation*}

\noindent is a 2nd order mock theta function (see \cite{GM1}). To see this, replace $q$ and $b$ by $q^4$ and $q$, respectively in \eqref{l31} to obtain
\begin{align*}
&\frac{(q^4;q^4)^2_\infty}{[1/q;q^4]_\infty}\sum_{n=-\infty}^\infty\frac{(-1)^nq^{n(2n+3)}}{1-q^{4n+1}}=
- \sum_{n=0}^\infty
\left(\frac{q^{4n+1}}{(1-q^{4n+1})^2}+\frac{q^{4n+3}}{(1-q^{4n+3})^2}\right)+ \sum_{\substack{n=-\infty \\ n \neq 0}}^{\infty} \frac{(-1)^nnq^{2n(n+1)}}{1-q^{4n}}.
\end{align*}

\noindent By (\cite[(3.2a), (3.2b), Eq. (5.2)]{Hi-Mo1}), we have

\begin{equation*}
B(q)=\frac{1}{(q,q^3,q^4;q^4)_\infty}\sum_{n=-\infty}^\infty\frac{(-1)^nq^{n(2n+3)}}{1-q^{4n+1}}
\end{equation*}

\noindent and thus \eqref{r} follows.

 Extracting the coefficient of $q^n$ on both sides of \eqref{r}, we obtain the following corollary.

\begin{corollary} For a fixed positive integer $n$, we have
\[
\sum_{\substack{m\in \mathbb{Z}^+ \\ 2m^2+m+1 \leq n}} (-1)^m(2m+1)c(B, n-2m^2-2m-1)
=\sum_{\substack{0<d|n \\ \frac{n}{d} \textrm{ odd}}} d
-\sum_{\substack{a,b\in\mathbb{Z}^+ \\ 1\leq a < b \\2ab=n\\a \not\equiv b \pmod{2}}} (-1)^aa.
\]
\end{corollary}

Similar results exist, for example, for the mock theta functions $\psi(q)$, $\rho(q)$ and $\lambda(q)$ of order 6 and $V_{0}(q)$ of order 8 as they can be written in terms of Appell-Lerch series (see Section 5 of \cite{Hi-Mo1}).

\section*{Acknowledgements}
The first author was partially supported by the Singapore Ministry
of Education Academic Research Fund, Tier 2,  project number
MOE2014-T2-1-051, ARC40/14. The second author was partially supported
by National Natural Science Foundation of China (Grant No. 11501398),
Natural Science Foundation of Jiangsu Province (Grant No. BK20150304) and
 Natural Science Foundation of the Jiangsu Higher Education
Institutions of China (Grant No. 15KJB110020).
The third author would like to thank the Institut des Hautes {\'E}tudes Scientifiques for their support
during the preparation of this paper. This material is based upon work supported by the National Science Foundation
under Grant No. 1002477. Finally, the authors would like to thank Jeremy Lovejoy for his motivating remark in \cite{irr}.

\end{document}